\documentclass{amsart}

\usepackage[utf8]{inputenc}
\usepackage{amsmath, amssymb, amsfonts, amsthm}
\usepackage[english]{babel}
\usepackage[all]{xy}
\usepackage{hyperref}
\usepackage{enumerate}

\usepackage{tikz}

\DeclareFontFamily{U}{wncy}{}
\DeclareFontShape{U}{wncy}{m}{n}{<->wncyr10}{}
\DeclareSymbolFont{mcy}{U}{wncy}{m}{n}
\DeclareMathSymbol{\sha}{\mathord}{mcy}{"58}

\usetikzlibrary{graphs}
\usetikzlibrary{graphs.standard}

\theoremstyle{plain}
\newtheorem{theorem}{Theorem}[section]

\newtheorem{proposition}[theorem]{Proposition}
\newtheorem{lemma}[theorem]{Lemma}

\theoremstyle{definition}
\newtheorem{definition}[theorem]{Definition}
\theoremstyle{remark}
\newtheorem{remark}[theorem]{Remark}

\numberwithin{equation}{section}

\DeclareMathOperator{\Pic}{Pic}
\DeclareMathOperator{\Gal}{Gal}
\DeclareMathOperator{\PGL}{PGL}
\DeclareMathOperator{\GL}{GL}
\DeclareMathOperator{\Aut}{Aut}
\DeclareMathOperator{\Tr}{Tr}

\def\F{\mathbf{F}}
\def\G{\mathbf{G}}
\def\P{\mathbf{P}}
\def\Z{\mathbf{Z}}

\def\E{\mathbf{E}}

\def\C{\mathcal{C}}

\def\O{\mathcal{O}}
\def\B{\mathcal{B}}

\newcommand{\Id}{\textrm{Id}}
\newcommand{\map}[4]{\left\{ \begin{array}{ccc}
                                #1 & \longrightarrow & #2
                               \\ #3 & \longmapsto     & #4
                             \end{array}\right.}
\newcommand{\ratmap}[4]{\left\{ \begin{array}{ccc}
                                #1 & \dashrightarrow & #2
                               \\ #3 & \longmapsto     & #4
                             \end{array}\right.}
\renewcommand{\leq}{\leqslant}
\renewcommand{\geq}{\geqslant}
\title{Anticanonical codes from del Pezzo surfaces with Picard rank one}

\author[R. Blache]{Régis Blache}
\address{LAMIA, Université des Antilles}
\email{regis.blache@univ-antilles.fr}
\author[A. Couvreur]{Alain Couvreur}
\address{INRIA \& Laboratoire LIX, CNRS UMR 7161,\'Ecole Polytechnique}
\email{alain.couvreur@inria.fr}
\author[E. Hallouin]{Emmanuel Hallouin}
\address{Institut de Math\'ematiques de Toulouse, UMR 5219}
\email{hallouin@univ-tlse2.fr}
\author[D. Madore]{David Madore}
\address{LTCI Telecom ParisTech}
\email{david.madore@telecom-paristech.fr}
\author[J. Nardi]{Jade Nardi}
\address{Institut de Math\'ematiques de Toulouse, UMR 5219}
\email{Jade.Nardi@math.univ-toulouse.fr}
\author[M. Rambaud]{Matthieu Rambaud}
\address{LTCI Telecom ParisTech}
\email{matthieu.rambaud@telecom-paristech.fr}
\author[H. Randriam]{Hugues Randriam}
\thanks{The authors have been funded by ANR project ANR-15-CE39-0013 Manta}
\address{LTCI Telecom ParisTech}
\email{randriam@telecom-paristech.fr}

\begin{document}

\begin{abstract}
  We construct algebraic geometric codes from del Pezzo surfaces and
  focus on the ones having Picard rank one and the codes associated to
  the anticanonical class. We give explicit constructions of del Pezzo
  surfaces of degree $4, 5$ and $6$, compute the parameters of the
  associated anticanonical codes and study their isomorphisms arising
  from the automorphisms of the surface.
  We obtain codes with excellent parameters and some of them turn out
  to beat the best known codes listed on the database {\em codetable}.
\end{abstract}



\date{\today}

\maketitle

\markleft{\sc Blache, Couvreur, Hallouin, Madore, Nardi, Rambaud \& Randriam}

\section{Introduction}

The aim of this article is to construct codes from algebraic
surfaces. Codes coming from algebraic geometry have seen a growing
interest since their original contruction by Goppa.
We consider evaluation codes: given an algebraic variety $X$
defined over a finite field $\F_q$, and an effective divisor $D$ on
it, we evaluate the global sections in $H^0(X,\O_X(D))$ at the
rational points of $X$. We get a code of length $\sharp X(\F_q)$, the
number of rational points, and of dimension $h^0(X,\O_X(D))$ if the
evaluation map is injective. It remains to determine the minimum
distance.

Most of the algebraic geometry codes studied in the literature are
based on algebraic curves, since Riemann-Roch theorem for curves is a
powerful tool to estimate the dimension of the space of functions with
given poles, and the number of zeroes of such a function never exceeds
its number of poles.  The study of evaluation codes over higher
dimensional varieties is more difficult. On one hand, Riemann-Roch
theorem involves the dimensions of the higher cohomology spaces
$H^i(X,\O_X(D))$. On the other hand, one has to bound the maximal
number of rational points of the schemes defined by global sections of
$\O_X(D)$ in order to estimate the minimum distance. Such a scheme is
no longer zero-dimensional and might be singular or reducible. As a
consequence one has to use new tools for the estimation of its number
of rational points.

In the case of surfaces, one first has to determine the intersection
pairing. Then, Riemann-Roch theorem often gives a lower bound for the
dimension of the space of global sections. We want to use adjunction
formula, to get the arithmetic genus of the curves coming from the
zeroes of such a section, then to apply the Hasse-Weil-Serre bound on
the number of points of irreducible curves. In order to do this, one
also has to control their decomposition into irreducible components.

This can be done working in the Picard lattice of the surface $X$. In
this group, a decomposition of the zero scheme of a global section of
$D$ corresponds to writing the class of $D$ as a sum of classes of
effective divisors. A very efficient way to control this is to ask the
Picard lattice to have rank $1$ \cite[Lemma 2.1]{zarzar}.

We shall work on del Pezzo surfaces. A del Pezzo surface of degree $d$
over an algebraically closed field is obtained from the projective
plane $\P^2$ by blowing up $9-d$ points in general position (see
Section~\ref{delpezzo}). As a consequence, such surfaces are close to
the plane; the arithmetic and geometric theory of these objects has
been thoroughly studied (the classical reference is \cite{manin}) and
we can explicit all invariants above: the intersection pairing and the
(geometric) Picard lattice have fairly simple descriptions. They come
endowed with an action of the Frobenius automorphism, and the
(arithmetic) Picard lattice is the sublattice fixed by this action.

\subsection{Our contribution} We construct del Pezzo surfaces of
degree $4$, $5$ and $6$ having Picard rank $1$. In each case the
Picard lattice is generated by the canonical divisor $K_X$, and the
cone of effective divisors by the anticanonical divisor $-K_X$. We
call the code associated to this last divisor \emph{anticanonical}; it
has dimension $d+1$ (except for very small values of $q$), and all
zero schemes of global sections are irreducible with arithmetic genus
$1$ from the adjunction formula. As a consequence, they have at most
$q+1+\lfloor 2\sqrt{q}\rfloor$ points from the Hasse-Weil-Serre bound,
and this provides very good codes. Some of them turn out to beat the
best known codes listed in the database {\em codetables} \cite{code}.

A central role is played by the Frobenius action on the Picard lattice
of the surface, i.e. by the conjugacy class of the image of the Frobenius
in a certain Weyl group. It gives many properties of the surface, such as
the Picard rank or the number of rational points. Moreover, with the help of 
Galois cohomology, it enables us to determine the $\F_q$-rational 
automorphisms of the surface; since an automorphism of the
surface must preserve the (anti)canonical class, we deduce some
automorphisms of the codes.

Along the article, we try to be as constructive as possible. We give 
explicit descriptions of the anticanonical models. We also use
birational morphisms from our surfaces to the projective plane to
give explicit constructions of the codes, and Cremona
transformations to describe the automorphisms. 

% In a particular case, we improve the minimal distance, showing that
% the anticanonical linear system does not carry any maximal curve.

\subsection{Related works}
In some sense, this paper fills a gap in the study of algebraic
geometric codes constructed from del Pezzo surfaces (even if most of
the authors cited below do not mention the fact that they work on del
Pezzo surfaces). When $X$ is the projective plane $\P^2$ (the only del
Pezzo surface of degree $9$ over a finite field), the Picard group is
generated by the class $L$ of a line, and the evaluation code
associated to $\O(mL)$ is the well-known projective Reed-Muller code
of order $m$ (see \cite{La88}). There are two types of
del Pezzo surfaces of degree $8$, one having Picard rank $2$ (the
hyperbolic quadric, isomorphic to $\P^1\times\P^1$), and the other
having Picard rank $1$ (the elliptic quadric). Codes over these
surfaces have been studied by Edoukou \cite{edou}, and Couvreur and
Duursma \cite{couvduur}. In \cite[Section~3.2]{couv}, the second
author constructs good codes coming from del Pezzo surfaces of degree
$6$ having Picard rank $2$. In \cite{ls}, Little and Schenck consider
anticanonical codes on del Pezzo surfaces of degree $3$ and $4$ having
Picard rank $1$. Boguslavsky \cite{bogus} gives the parameters of
anticanonical codes on \emph{split} del Pezzo surfaces, i.e. on
surfaces having maximal Picard rank.

Let us finally mention some works on codes on other blowups of the
plane \cite{davis,balli1}. These blowups are no longer del Pezzo
surfaces: since the blown up points lie on one or two lines, they are
not in general position. Moreover the evaluation set is not the set of
rational points of the surface, but the torus $\G_m^2(\F_q)$.

\subsection{Outline of the article} The paper is organized as
follows. In Section~\ref{delpezzo}, we recall the necessary material
on del Pezzo surfaces, and give the classification of such surfaces of
degrees $5$ and $6$ over a finite field. Then we construct degree $6$
del Pezzo surfaces in Section~\ref{delpezzo6}; we give the parameters
of the associated anticanonical codes, and determine the automorphisms
of the surfaces. We construct del Pezzo surfaces of degree $5$ and the
corresponding codes in Section~\ref{delpezzo5}. %  We give two geometric
% constructions of these codes (as evaluating on one hand linear forms
% at the rational points of a surface in $\P^5$, on the other hand
% particular quintics at some rational points of the projective plane),
% then we determine the automorphism group of the surface.
We give two geometric constructions of these codes: the first one by
evaluating linear forms at the rational points of a surface embedded
in $\P^5$ and the second one by evaluating particular quintics forms
at some rational points of the projective plane. Then, we determine
the automorphism group of the surface.  Finally, we construct del
Pezzo surfaces of degree $4$ in Section~\ref{degre4}, and
determine the parameters of the associated codes.

\subsection*{Acknowledgements}
The authors would like to thank Markus Grassl for pointing out the
existence of automorphisms of the codes.

\section{Context and notation}
In the following, we fix $\F_q$ a finite field of characteristic $p$
and an algebraic closure $\overline{\F}_q$.  We denote by $\sigma$ a
generator of the absolute Galois group
$G:=\Gal(\overline{\F}_q/\F_q)$.  The projective space of dimension
$r$ over $\F_q$ is denoted by $\P^r$.  On a surface $X$, the
intersection product of two divisor classes $A, B$ is denoted by
$A \cdot B$ and the self intersection $A \cdot A$ of $A$ is denoted by
$A ^{\cdot 2}$. Given an effective divisor $A$, the complete linear
system associated to $A$ is denoted by $|A|$, the Picard group of $X$
is denoted by $\textrm{Pic}(X)$ and the canonical class of $X$ is
denoted by $K_X$.

Given a smooth projective geometrically connected surface $X$ over
$\F_q$ and divisor $D$ on $X$, the code $\C(X(\F_q), D)$ is defined as the
image of the map
\[
  \map{H^0(X, \mathcal O_X(D))}{\F_q^n}{f}{(f(P))_{P\in X(\F_q)}.} 
\]
Note that for the map to be well--defined, one needs to order the
rational points of $X$, which can be done arbitrarily since the choice
of another order would provide an isometric code with respect to the
Hamming metric. Similarly, the evaluation of a global section of
$H^0(X, \mathcal O_X (D))$ at a point $P$ depends on the choice of a
generator of the stalk of the sheaf at $P$ but choosing another system
of generators would provide an isometric code. Since we are mostly
interested in the parameters of the code : its dimension and minimum
distance, it is sufficient to consider our code up to isometry.

\section{del Pezzo surfaces}
\label{delpezzo}

In this section we collect some definitions and well known facts about
del Pezzo surfaces, that we shall use in the sequel. For the proofs
and many other results, we refer the reader to \cite[Chapter 24
sq.]{manin}.

\begin{definition}
A smooth projective surface $X$ defined over a field $k$ is \emph{del
  Pezzo} when its anticanonical divisor $-K_X$ is ample. Its
\emph{degree} is the self-intersection number $d:=K_X^{\cdot 2}$.  
\end{definition}

We assume $3\leq d\leq 7$ in the following. We know from
\cite[Theorems 24.4, 24.5]{manin} that such a surface is isomorphic
(over the algebraic closure $\overline{k}$) to the blow-up of the
projective plane $\P^2$ at $r:=9-d$ points in \emph{general position},
i.e. such that no three of them are collinear, and no six lie on a
conic. In this case, the anticanonical divisor is very ample; the
space of its global sections has dimension
\[
  \dim H^0(X,\mathcal O(-K_X))=d+1=10-r,
\]
and it defines an embedding of $X$ into
$\P^d$, whose image has degree $d$. The image of this embedding is
called the \emph{anticanonical model} of $X$.

If $X'\rightarrow X$ is a birational morphism, and $X'$ is a del Pezzo
surface, then $X$ is a del Pezzo surface from \cite[Theorem 24.5.2
(i)]{manin}. 

Let $X$ be as above; it is isomorphic (over $\overline{k}$) to the
blowup $\pi$ of the projective plane at $p_1,\ldots,p_r$. Let
$E_0:=\pi^\ast H$ denote the pullback of the class of a line in $\P^2$
%(that generates $\Pic(\P^2)$)
and $E_1, \ldots, E_r$, denote respectively the class of the
exceptional divisor of the blowup at $p_1, \ldots, p_r$. From
\cite[Theorems 25.1]{manin} (see also \cite[Proposition V.3.2]{hart}),
we know that the geometric Picard lattice $\Pic(X\otimes\overline{k})$
is free of rank $r+1$, with basis $E_0,\ldots,E_r$. The intersection
pairing is defined by
$$E_0^{\cdot 2}=1,~E_i^{\cdot 2}=-1,~i\geq 1,~E_i\cdot E_j=0,~i\neq j$$
and the class of the canonical divisor is
$$K_X=-3E_0+\sum_{i=1}^r E_i.$$

In the following, we assume that $X$ is defined over $\F_q$.
Associated to $X$ is a representation of $G$ on
$\Pic(X\otimes \overline{\F}_q)$ that respects the intersection
pairing and the canonical divisor. If
$\sigma^\ast\in\Aut(\Pic(X\otimes \overline{\F}_q))$ is the image of
$\sigma$ under this representation, then we know from a result of Weil
\cite[Theorem 27.1]{manin} that the number of rational points of $X$
is given by
$$\sharp  X(\F_q)=q^2+q\Tr(\sigma^\ast)+1.$$
Moreover, since $X$ is projective and smooth and the ground field is
finite, we have $\Pic(X)=\Pic(X\otimes \overline{\F}_q)^G$, and the
Picard rank is the multiplicity of the eigenvalue $1$ for
$\sigma^\ast$.

Following Dolgachev \cite[Section~8.2]{dolga}, we define the lattice
$\E_r$ as the orthogonal $K_X^\perp$ of the canonical divisor in
$\Pic(X\otimes \overline{\F}_q)$. This is a root lattice, whose Weyl
group is denoted $W(\E_r)$. The image of the above Galois
representation lies in $W(\E_r)$, and it is a finite quotient of $G$,
thus a cyclic subgroup.

Following Manin \cite{manin}, we define the \emph{type} of the del Pezzo surface $X$ as the conjugacy class of the element $\sigma^\ast$ in $W(\E_r)$.

\subsection{Isomorphism classes of del Pezzo surfaces of degrees five and six over a finite field}

For a del Pezzo surface of degree $5$ or $6$ over a perfect field, the
types of Manin are actually isomorphism classes (see \cite[Theorem
3.1.3]{skoro} for degree $5$, degree $6$ is treated in Section
\ref{aut6}). When the field $k$ is finite, we deduce that the
isomorphism classes of such del Pezzo surfaces correspond to the
conjugacy classes of elements in the Weyl group $W(\E_r)$.

We deduce the following classification for degree $5$ or $6$ del Pezzo
surfaces over finite fields \cite[Section~8.2]{dolga}, \cite[Section
3]{trepa}.

\subsubsection{Del Pezzo surfaces of degree $6$} When $d=6$, the root lattice is $\E_3=A_1 + A_2$, and its Weyl group
is
$$W(A_1)\times W(A_2)\simeq \mathfrak{S}_2\times \mathfrak{S}_3\simeq \mathfrak D_{12},$$
where $\mathfrak S_i$ denotes the symmetric group on $i$ letters and
$\mathfrak D_{12}$ denotes the dihedral group of order $12$, generated by
the symmetries $s_i:=s_{\alpha_i}$ with respect to the following roots
which form a principal system:

\[
    \alpha_1 =E_0-E_1-E_2-E_3;~ \alpha_2 =E_1-E_2;~ \alpha_3 =E_2-E_3.
\]
We have $s_1s_i=s_is_1$ for $i>1$ and
$(s_2s_3)^3=\Id$. Table~\ref{tab:degree6} summarizes the different
types of Del Pezzo surfaces with the corresponding Weyl conjugation
classes.

\begin{table}[!h]
  \centering
  \begin{tabular}{|c|c|c|c|c|}
\hline
Type & Weyl classes & Eigenvalues of $\sigma^\ast$ & $\Tr(\sigma^*)$ & Picard rank \\
\hline
\hline
$6_1$ & $\{Id\}$ & $1,1,1,1$ & 4 & 4 \\
\hline
$6_2$ & $\{s_1\}$ & $1,1,1,-1$ & 2 & 3\\
\hline
$6_3$ & $\{s_2,s_3,s_2s_3s_2\}$ & $1,1,1,-1$ & 2 & 3\\
\hline
$6_4$ & $\{s_2s_3,s_3s_2\}$ & $1,1,\jmath,\overline{\jmath}$ & 1 & 2\\
\hline
$6_5$ & $\{s_1s_2,s_1s_3,s_1s_2s_3s_2\}$ & $1,1,-1,-1$ & 0 & 2\\
\hline
$6_6$ & $\{s_1s_2s_3,s_1s_3s_2\}$ & $1,-1,\jmath,\overline{\jmath}$ & -1 & 1\\
\hline    
  \end{tabular}
  \bigskip
  \caption{Types of degree $6$ del Pezzo surfaces}
  \label{tab:degree6}
\end{table}

\subsubsection{Del Pezzo surfaces of degree $5$}
When $d=5$, the root lattice is $\E_4=A_4$, and its Weyl group is
$W(A_4)\simeq \mathfrak{S}_5$ the symmetric group on $5$ letters, generated by
the symmetries $s_i:=s_{\alpha_i}$, $1\leq i\leq 4$, with respect to
the roots
\[ 
\alpha_1=E_0-E_1-E_2-E_3; ~\alpha_2=E_1-E_2;~ \alpha_3=E_2-E_3;~ \alpha_4=E_3-E_4.   
\]
We identify these symmetries
respectively to the transpositions $(4,5),(1,2),(2,3)$ and $(3,4)$. The
conjugacy classes in $\mathfrak{S}_5$ identify to the partitions of $5$, and we
get the classification summarised in Table~\ref{tab:degree5}.

\begin{table}[!b]
  \centering
\begin{tabular}{|c|c|c|c|c|}
\hline
Type & Weyl classes & Eigenvalues of $\sigma^\ast$ & $\Tr (\sigma^*)$ & Picard rank \\
\hline
\hline
$5_1$ & $\{1,1,1,1,1\}$ & $1,1,1,1,1$ & 5 & 5\\
\hline
$5_2$ & $\{2,1,1,1\}$ & $1,1,1,1,-1$ & 3 & 4\\
\hline
$5_3$ & $\{3,1,1\}$ & $1,1,1,\jmath,\overline{\jmath}$ & 2 & 3\\
\hline
$5_4$ & $\{2,2,1\}$ & $1,1,1,-1,-1$ & 1 & 3 \\
\hline
$5_5$ & $\{3,2\}$ & $1,1,-1,\jmath,\overline{\jmath}$ & 0 & 2\\
\hline
$5_6$ & $\{4,1\}$ & $1,1,-1,\imath,-\imath$ & 1 & 2\\
\hline
$5_7$ & $\{5\}$ & $1,\zeta_5,\zeta_5^2,\zeta_5^3,\zeta_5^4$ & 0 & 1\\
\hline
\end{tabular}
\bigskip
  \caption{Types of degree $5$ del Pezzo surfaces}
  \label{tab:degree5}
\end{table}

We see that there exists exactly one isomorphism class of del Pezzo
surface having Picard rank $1$ in each degree $5$ or $6$,
corresponding respectively to the types $6_6$ and $5_7$. We construct
explicitely surfaces of these two types in the following sections.

We end this section with a technical result, that will be useful when
we estimate the minimum distance of the codes. It is close to
\cite[Theorem 3.3]{ls}.

\begin{lemma}
\label{fpc}
Assume that $X$ is a del Pezzo surface over $\F_q$, with
$\Pic(X)=\Z K_X$. If $C\in |-K_X|$ is an anticanonical curve that is
not absolutely irreducible, then we have $\sharp  C(\F_q)\leq 2$.
\end{lemma}

\begin{proof}
  Let $C=C_1\cup\cdots\cup C_r$ be the decomposition of $C$ into
  absolutely irreducible components. Since $C\in |-K_X|$ and $-K_X$
  generates the Picard group of $X$, $C$ is irreducible over
  $\F_q$. We deduce that we must have $r\geq 2$, and the components
  $C_i$ are cyclically permuted by $G$. As a consequence, every
  rational point must be at the intersection of all the $C_i$'s.

Assume that $C$ contains a rational point; from what we have just
said, we must have $C_i\cdot C_j\geq 1$ for any $i\neq j$. In the
Picard group, we get $-K_X=C_1+\cdots+C_r$, and the arithmetic genera
satisfy \cite[Ex V.1.3]{hart}
$$1=\pi(-K_X)=\sum_{i=1}^r\pi(C_i)+\sum_{1\leq i<j\leq r} C_i\cdot C_j-(r-1)$$
Since the $C_i$'s are absolutely irreducible, their arithmetic genera are
nonnegative and hence, we get
\[
  r(r-1)/2\leq\sum_{1\leq i<j\leq r} C_i\cdot C_j\ \leq r, \quad {\rm and}\quad
r\leq 3.
\]
If $r=2$, then $C_1\cdot C_2\leq 2$ and $C$ contains at most
two rational points.
If $r=3$, then $C_i\cdot C_j=1$ for any $i\neq j$, and
$C$ contains at most one rational point.
\end{proof}

\section{Anticanonical codes on some degree six del Pezzo surfaces}
\label{delpezzo6}

In this section, we construct some degree six del Pezzo surfaces with
Picard rank one over any finite field, then we determine the
parameters of the anticanonical codes on these surfaces.

\subsection{Construction of the surface}

We choose
\begin{itemize}
\item $p_1,p_2=p_1^\sigma$ two conjugate points in
  $\P^2(\F_{q^2})\backslash\P^2(\F_{q})$,
	\item $p_3,p_4=p_3^\sigma,p_5=p_3^{\sigma^2}$ three conjugate
      points in $\P^2(\F_{q^3})\backslash\P^2(\F_{q})$.
\end{itemize}

Recall that five points in the projective plane are in general
position when no three of them are collinear. For any prime power $q$
it is possible to choose $p_1,\ldots,p_5$ as above in general
position: from \cite[Lemma 2.4]{bfl} it is sufficient to choose them
on a smooth conic; such a curve exists for any $q$. Now if $C$ denotes
a smooth conic in the projective plane, we have
$\sharp C(\F_{q^i})=q^i+1$ for any $i\geq 1$, and $C$ has
$q^3-q\geq 6$ points defined over $\F_{q^3}$ but not over $\F_q$, and
$q^2-q\geq 2$ points defined over $\F_{q^2}$ but not over $\F_q$.

Let $\widetilde{X}$ denote the surface obtained from $\P^2$ after
blowing up the points $p_1,\ldots,p_5$, $E_1,\ldots,E_5$ the
corresponding exceptional divisors, and
$\pi:\widetilde{X}\mapsto \P^2$ the composition of the five
blowups. Since $\{p_1,\ldots,p_5\}$ is stable under the action of $G$,
the map $\pi$ and the surface $\widetilde{X}$ are defined over $\F_q$.

Applying the results in Section~\ref{delpezzo} to $\widetilde{X}$, we
get the following properties. It is a degree $4$ del Pezzo surface
with geometric Picard lattice
$\Pic(\widetilde{X}\otimes\overline{\F}_q)=\oplus_{i=0}^5 \Z E_i$, and
canonical class $K_{\widetilde{X}}=-3E_0+\sum_{i=1}^5 E_i$. From our
choice of the $p_i$'s, the map $\sigma^\ast$ acts on the $E_i$'s by the
permutation $(E_0)(E_1E_2)(E_3E_4E_5)$. As a consequence, the Picard
lattice of $\widetilde{X}$ is
\[
  \Pic(\widetilde{X})= \Z E_0+\Z(E_1+E_2)+\Z(E_3+E_4+E_5),
\]and we get
$\Tr(\sigma^\ast)=1$. Thus $\widetilde{X}$ has $q^2+q+1$ points.

In the following, we denote by $L$ the line $(p_1p_2)$, and by $C$ the
conic passing through $p_1,\ldots,p_5$ in $\P^2$; note that $C$ is
unique since we assumed the points to be in general position, and both
curves are defined over $\F_q$ from our choice of the $p_i$'s. Let
$\widetilde{L}$ and $\widetilde{C}$ denote the respective strict
transforms of $L$ and $C$ in $\widetilde{X}$; they are irreducible
curves defined over $\F_q$. Their images in $\Pic(\widetilde{X})$
satisfy
$$\widetilde{L}=E_0-E_1-E_2,~\widetilde{C}=2E_0-\sum_{i=1}^5 E_i.$$
We get $\widetilde{L}^{\cdot 2}=\widetilde{C}^{\cdot 2}=-1$, and they
have arithmetic genus zero from the adjunction formula. Moreover they
are disjoint since $\widetilde{L}\cdot \widetilde{C}=0$.

From Castelnuovo's contractibility criterion \cite[Theorem
3.30]{bade}, \cite[Theorem 21.5]{manin}, there exists a smooth
projective surface $X$, and a birational morphism
$\psi:\widetilde{X}\rightarrow X$ contracting $\widetilde{L}$ and
$\widetilde{C}$ to points $l=\psi(\widetilde{L})$ and
$c=\psi(\widetilde{C})$. In other words, the map $\psi$ is the
composition of the blowups of $X$ at $l$ and $c$, and $\widetilde{L}$,
$\widetilde{C}$ are the corresponding exceptional divisors. We have
the diagram

\begin{equation*}
\xymatrix{
& \widetilde{X} \ar@{->}_\pi[dl] \ar@{->}^\psi[dr]& \\
\P^2 &  & X\\
}
\end{equation*}

\begin{lemma}
\label{Picgeom}
The geometric Picard lattice $\Pic(X\otimes\overline{\F}_q)$
identifies to the sublattice of
$\Pic(\widetilde{X}\otimes\overline{\F}_q)$ generated by the classes
\begin{align*}
  F_0&=3E_0-2E_1-\sum_{i=2}^5E_i;\\
  F_1&=E_0-E_1-E_3;\\
  F_2&=E_0-E_1-E_4;\\
  F_3&=E_0-E_1-E_5,
\end{align*}
which satisfy $F_0^{\cdot 2}=-F_i^{\cdot 2}=1$, $i\geq 1$, and
$F_i\cdot F_j =0$, $i\neq j$.
\end{lemma}

\begin{proof}
We use \cite[Proposition V.3.2]{hart}.
Since $\psi$ is the composition of two blowups, the map $\psi^\ast$
from $\Pic(X\otimes\overline{\F}_q)$ to
$\Pic(\widetilde{X}\otimes\overline{\F}_q)$ is an isometry for the
intersection pairing; as this pairing is non degenerate, it is an
injection, and $\Pic(X\otimes\overline{\F}_q)$ identifies to its
image.

The classes of the exceptional divisors are $\widetilde{L}$ and
$\widetilde{C}$; as a consequence, the image
$\Pic(X\otimes\overline{\F}_q)$ is the orthogonal of
$\Z\widetilde{L}+\Z\widetilde{C}$ in
$\Pic(\widetilde{X}\otimes\overline{\F}_q)$.
A class $a_0E_0-\sum_{i=1}^5 a_iE_i$ is in the orthogonal
if, and only if
\[a_0=a_1+a_2=a_3+a_4+a_5.
\]
When we look for classes having self-intersection $\pm 1$, we get the
condition
\[a_0^2-\sum_{i=1}^5 a_i^2=\pm 1.
\]
We get the class $F_1$,
and any class orthogonal to $F_1$ must satisfy $a_0=a_1+a_3$. The
class $F_2$ satisfies the three equalities above, and orthogonality
adds the equation $a_0=a_1+a_4$. We get $F_3$, and the last equation
$a_0=a_1+a_5$. The five equations above finally give $F_0$.
The classes $F_0,\ldots,F_3$ clearly form a basis of
$\Pic(X\otimes\overline{\F}_q)$ from their intersection products.
\end{proof}

We determine the canonical divisor class of $X$.

\begin{lemma}
  In the Picard lattice of $X$, the class of the canonical divisor is
  \[K_X=-3F_0+F_1+F_2+F_3.\]
\end{lemma}

\begin{proof}
  From \cite[Proposition V.3.3]{hart}, we have
  $K_{\widetilde{X}}=\psi^\ast K_X+\widetilde{L}+\widetilde{C}$ in
  $\Pic(\widetilde{X})$. We get
  $\psi^\ast K_X=-6E_0+3(E_1+E_2)+2(E_3+E_4+E_5)$, and the result
  comes from the identification of the lemma above.
\end{proof}

We are ready to prove the main properties of our surface.

\begin{proposition}
\label{ratX}
The surface $X$ has the following properties
\begin{enumerate}[(i)]
	\item it is a degree $6$ del Pezzo surface;
	\item\label{item:ii} its Picard lattice $\Pic(X)$ has rank one, and is
      generated by $K_X$;
	\item\label{item:iii} it is defined over $\F_q$, and has $q^2-q+1$ rational
      points.
\end{enumerate}
\end{proposition}

\begin{proof}
  As $\psi$ is a birational morphism from a del Pezzo surface, we know
  from Section~\ref{delpezzo} that $X$ is a del Pezzo surface. Its
  degree is $K_X^{\cdot 2}=6$.

  Recall that we have $\Pic(X)=\Pic(X\otimes\overline{\F}_q)^G$. From
  the description of the action of $\sigma^\ast$ on the $E_i$'s, we
  deduce that $D=\sum_{i=0}^3 a_iF_i$ satisfies $D^\sigma=D$ if, and
  only if we have $a_1=a_2=a_3$ and $a_0=-3a_1$, i.e. $D=a_1K_X$. This
  proves assertion (\ref{item:ii}).

  Finally, $X$ is defined over $\F_q$ since the canonical class is. To
  compute the number of its rational points, we use Weil's result from
  Section~\ref{delpezzo}. The matrix of the action of Frobenius on
  $\Pic(X\otimes\overline{\F}_q)$, with respect to the basis
  $\{F_0,F_1,F_2,F_3\}$ is
\begin{equation}\label{Action_Frob_Deg-6}
%$$   %%% Ajout e
M = \left(\begin{array}{cccc}
2 & 1 & 1 & 1 \\
-1 & -1 & -1 & 0\\
-1 & 0 & -1 & -1 \\
- 1 & -1 & 0 & -1 \\
\end{array}\right)
%$$
\end{equation}
whose trace is $-1$.
\end{proof}

\subsection{Anticanonical codes}
Here we determine the parameters of the evaluation code
$\C(X(\F_q),-K_X)$. %  obtained by evaluating the global sections of
% the sheaf $\O(-K_X)$ at the rational points of $X$.

Its length is $\sharp X(\F_q)=q^2-q+1$ from Proposition \ref{ratX}
(\ref{item:iii}).  The dimension of the space of global sections is
$h^0(X,-K_X)=d+1=7$ from Section~\ref{delpezzo}. We will see below
that the evaluation map is injective for any $q\geq 4$, and in this
case we get a code of dimension $7$.

Let $D\in |-K_X|$. Since $-K_X$ generates the Picard lattice of $X$,
the curve $D$ is irreducible over $\F_q$. If it is absolutely
irreducible, the adjunction formula gives
\[
  2p_a(D)-2=D\cdot(D+K_X)=-K_X\cdot(-K_X+K_X)=0.
\]
Thus, $D$ has
arithmetic genus $1$ and contains at most
$q+1+\lfloor 2\sqrt{q}\rfloor$ rational points.

If~$D \in \left|-K_X\right|$ is irreducible but not absolutely
irreducible, we apply Lemma~\ref{fpc}. Hence,
the maximal number of points of a section comes from an absolutely
irreducible one, and the minimum distance is at least
$q^2-2q-\lfloor 2\sqrt{q}\rfloor$. This number is positive as long as
$q\geq 4$, and we get

\begin{proposition}
\label{code6}
Assume we have $q\geq 4$. Then the code $\C(X,-K_X)$ has parameters
\[
  [q^2-q+1,7,\geq q^2-2q-\lfloor 2\sqrt{q}\rfloor].
\]
\end{proposition}

Parameters of the codes for small values of $q$ are summarized in
Table~\ref{tab:param6}. For $q\in\{5,7,8,9\}$ these codes attain
the parameters of the best known codes listed in {\em codetable}
databse\cite{code}.
%For the little values of $q$ we get
\begin{table}[!h]
  \centering
  \begin{tabular}{|c|c|c|c|c|c|}
\hline
$q$ & 4 & 5 & 7 & 8 & 9 \\
\hline
     $[n,k,d]$ & $[13,7, \geq 4]$ & $[21,7,\geq 11]$ & $[43,7,\geq 30]$ &
                  $[57,7,\geq 43]$ & $[73,7,\geq 57]$\\
\hline    
  \end{tabular}
  \medskip
  \caption{Parameters of codes from our degree 6 del Pezzo surface}
  \label{tab:param6}
\end{table}

\begin{remark}
\label{rema}
For
$q=4$, {\tt magma} \cite{mag} calculations give a minimum distance
equal to $5$, instead of $4$, for all (random) choices of the points
$p_i$ we made. Actually the anticanonical system does not carry any
maximal elliptic curve, as we shall prove in Section~\ref{q4}. This is no
longer true when $q\geq
5$: the minimum distances we observe are those given in the above
Proposition.

\end{remark}

\subsection{Automorphisms of the surface}
\label{aut6}

Asking {\tt magma} for the automorphism groups of these codes, we get
respectively groups of order $234,504,1548$ for $q=4,5,7$. Among these
are the $q-1$ multiplications by scalars. The aim of this section is
to show that the remaining $6(q^2-q+1)$ automorphisms come from
automorphisms of the surface.

We first describe the geometric group of automorphisms of the split
degree $6$ del Pezzo surface $X_0$ obtained by blowing up the
projective plane at the points $(1:0:0)$, $(0:1:0)$ and
$(0:0:1)$. Note that $X_0$ is a toric variety whose maximal torus we
note $\G_m^2$. From \cite[Theorem 8.4.2]{dolga}, we have
$$\Aut(X_0)=\mathfrak D_{12}\ltimes (\overline{\F}_q^\ast)^2$$
where $\mathfrak D_{12}$ is the dihedral group of order $12$, which is the Weyl
group $W(\E_3)$ of isometries of the geometric Picard lattice of a
degree $6$ del Pezzo surface.

We first describe the action of these automorphisms on the maximal
torus $\G_m^2$; since the group $\mathfrak D_{12}$ is generated by the
permutations (of the projective coordinates) $g_1=(12)$, $g_2=(123)$,
and the standard quadratic transform $g_3$, we get
$$g_1(x,y)=(y,x),~g_2(x,y)=\left(\frac{y}{x},\frac{1}{x}\right),~g_3(x,y)=\left(\frac{1}{x},\frac{1}{y}\right)$$
and $(a,b)\in (\overline{\F}_q^\ast)^2$ acts by $(a,b)(x,y)=(ax,by)$.

For any $g\in \mathfrak D_{12}$ and $(a,b)\in (\overline{\F}_q^\ast)^2$, set
$g\cdot(a,b):=g\circ(a,b)$ the composition of the above actions. One
easily verifies the following assertions 
 \begin{itemize}
 \item we have $g(a,b)=g\cdot (a,b)\cdot g^{-1}$, where $g(a,b)$ is the
   image of $(a,b)$ by the action of $g$ described above.
 \item the action of $\sigma$ on $\Aut(X_0)$ is given by
   $^\sigma(g\cdot (a,b))=g\cdot (a^q,b^q)$, since the surface $X_0$
   is split and the action of the Weyl group is defined over $\F_q$.
 \end{itemize}

 We now determine $H^1(G,\Aut(X_0))$, which gives the isomorphism
 classes of degree $6$ del Pezzo surfaces over $\F_q$ (all these
 surfaces become isomorphic over the algebraic closure
 $\overline{\F}_q$). We have a short exact sequence (with a trivial
 Galois action on the diedral group)
$$1\rightarrow (\overline{\F}_q^\ast)^2 \rightarrow \Aut(X_0) \rightarrow \mathfrak D_{12} \rightarrow 1$$
which is split by the map $g\mapsto g\cdot(1,1)$.
From \cite[I.5.5 Proposition 38]{serre}, we have a map
\[
  H^1(G,\Aut(X_0))\rightarrow H^1(G,\mathfrak D_{12})
\]
and, by functoriality of
$H^1$, the splitting above ensures us that this map is surjective. In
order to show that it is injective, it is sufficient to show that for
any twist $_g(\overline{\F}_q^\ast)^2$, the cohomology group
$H^1(G,\ \!\!\! _g(\overline{\F}_q^\ast)^2)$ vanishes \cite[I.5.5
Corollaire 2]{serre}. But this is a consequence of Lang's theorem
\cite[III.2.3 Théorème 1']{serre} since in any case the twist is a
smooth connected algebraic group.

\begin{remark}
  Note that, since the groups $\Aut(X_0)$ and $\mathfrak D_{12}$ are not abelian,
  the $H^1$'s are not groups but only pointed sets. For this reason,
  the injectivity of the map cannot be proved using Hilbert 90
  Theorem.
\end{remark}

We deduce that the set $H^1(G,\Aut(X_0))$ corresponds to the set of
conjugacy classes of elements of the group $\mathfrak D_{12}$.
With this at hand, we are ready to prove the following statement.

\begin{proposition}
\label{auto6}
There are $6(q^2-q+1)$ automorphisms of the surface $X$ defined over $\F_q$.
\end{proposition}

\begin{proof}
  The eigenvalues of the matrix $M$ of
  (\ref{Action_Frob_Deg-6}) are $1, -1, \jmath, \bar \jmath$. Therefore,
  this matrix has order $6$ and, from our calculation of
  $H^1(G,\Aut(X_0))$, the degree $6$ del Pezzo
  surface $X$ constructed above corresponds to an element of order $6$
  in $\mathfrak D_{12}$.  There is only one up to conjugacy, and we choose
  $\gamma:=g_2g_3$. In other words, if $\varphi:X_0\rightarrow X$ is
  an isomorphism over $\overline{\F}_q$, the cocyle corresponding to
  $X$, $c_X=c_\varphi$ sends $\sigma$ on $\gamma$.

  Let $h_X$ denote an automorphism of $X$ over $\overline{\F}_q$; from
  it we construct an automorphism
  $h=\varphi^{-1}\circ h_X \circ\varphi$ of $X_0$. Now $h_X$ is
  defined over $\F_q$ if and only if we have ${^\sigma}h_X=h_X$, that
  is
$${^\sigma}h=({^\sigma}{\varphi})^{-1}\circ {^\sigma}{h_X} \circ {^\sigma}{\varphi}=({^\sigma}{\varphi})^{-1}\circ h_X \circ {^\sigma}{\varphi}=c_X(\sigma)^{-1}\circ h \circ c_X(\sigma)=\gamma^{-1}\circ h\circ \gamma$$

If we write $h=g\cdot(a,b)$ as above, we get the condition

$$g\cdot(a^q,b^q)=\gamma^{-1} g\cdot(a,b) \gamma=\gamma^{-1} g \gamma \cdot\gamma^{-1}(a,b) \gamma=\gamma^{-1} g \gamma \cdot\left(b,\frac{b}{a}\right)$$
and the automorphisms of $X$ are the
$h_X=\varphi\circ g\cdot(a,b) \circ\varphi^{-1}$, where $g$ is in the
centralizer of $\gamma$, which is the order $6$ subgroup of $\mathfrak D_{12}$
generated by $\gamma$, and $(a,b)\in (\overline{\F}_q^\ast)^2$
satisfies $b=a^q$, $\frac{b}{a}=b^q$, i.e. $a^{q^2-q+1}=1$, $b=a^q$.
\end{proof}

\subsection{Improving the minimum distance over the field with
  four elements}
\label{q4}

As we observed in Remark \ref{rema}, the minimum distance of the
anticanonical code is one more than the bound given in Proposition
\ref{code6} when $q=4$. We prove this fact here.

Assume that the anticanonical linear system carries a maximal elliptic
curve $Z$ defined over $\F_4$, i.e. with $\sharp Z(\F_4)=N_4(1)=9$. This
curve must be smooth since its geometric genus equals its arithmetic
genus.

We begin with a lemma about the automorphisms of the surface $X$ and
their action on the curve $Z$.

\begin{lemma}
  \label{automax6}
  The automorphism group of $X$ satisfies the following properties.
\begin{enumerate}
\item The group $\Aut(X)$ contains an element of order $13$, which
permutes cyclically the set $X(\F_4)$.
\item There exists an $h\in \Aut(X)$ such that $h(Z)$ contains the points
$l$ and $c$.
\end{enumerate}
\end{lemma}

\begin{proof}
  The order of $\Aut(X)$ is $78=6\cdot 13$ from Proposition
  \ref{auto6}; thus this group contains an element $f$ of order $13$
  from a theorem of Cauchy. Since $\sharp X(\F_4)=13$, either $f$ permutes
  cyclically the rational points of $X$, or its fixes all of them.

  The automorphism $f$ preserves the exceptional divisors of $X$; as a
  consequence, it induces an automorphism on the complementary $U$ of
  these divisors in $X$. As we have seen in \ref{aut6}, the surface
  $X\otimes\overline{\F}_4$ is isomorphic to
  $X_0\otimes\overline{\F}_4$, and the image of $U$ under such an
  isomorphism is the maximal torus $\G_m^2$.

  Thus $f$ induces an automorphism of $X_0$, i.e. an element in
  $\mathfrak D_{12}\ltimes (\overline{\F}_4^\ast)^2$; since $f$ has order $13$,
  its image must lie in $(\overline{\F}_4^\ast)^2$ and have the form
  $(a,b)\neq (1,1)$. Since the automorphism $(x,y)\mapsto(ax,by)$ has
  no fixed point on the maximal torus, $f$ does not have any fixed
  point on $U$.

  Now the exceptional divisors of $X$ are the images under $\psi$ of
  the strict transforms of the lines $(p_ip_j)$, $1\leq i\leq 2$,
  $3\leq j\leq 5$. They are cyclically permuted by the action of the
  Galois group $\Gal(\overline{\F}_4/\F_4)$; any rational point on one
  of these divisors must lie on all, which does not happen. We get
  $X(\F_4)=U(\F_4)$, and $f$ has no fixed point in this set. This
  proves the first assertion.

  From above, since $c$ and $l$ lie in $X(\F_4)$, there exists some
  $1\leq i\leq 12$ such that $l=f^i(c)$; replacing $f$ by $f^i$ we
  assume that $l=f(c)$ in the following. If we write
  $Z(\F_4)=\{f^{i_1}(c),\ldots,f^{i_9}(c)\}$ for some
  $0\leq i_1<\ldots<i_9\leq 12$, then at least two of the $i_j$'s are
  consecutive, say $i_2=i_1+1$; the automorphism $h=f^{-i_1}$
  satisfies the requirements of the second assertion.
\end{proof}

Thus we can assume that $Z$ is a maximal curve containing $l$
and $c$. Denote by $\widetilde{Z}\subset \widetilde{X}$ its strict
transform under $\psi$; since $Z$ is smooth, it has multiplicity one
at $l$ and $c$, and we have
$\psi^\ast Z= \widetilde{Z}+\widetilde{L}+\widetilde{C}$ in
$\Pic(\widetilde{X})$. Moreover $\psi$ induces an isomorphism between
the curves $Z$ and $\widetilde{Z}$.

The curve $\widetilde{Z}$ lies in the anticanonical system of
$\widetilde{X}$; it is smooth and since for any $1 \leq i \leq 5$, we
have $Z \cdot E_i = (-K_X)\cdot E_i = 1$, then $Z$ is transversal to
the exceptional divisors $E_1,\ldots,E_5$. As a consequence, it is
isomorphic to its image $Y$ under $\pi$, which is a smooth cubic
passing through the points $p_1,\ldots,p_5$.

Thus the curve $Y$ is a smooth elliptic curve in $\P^2$ having $9$
rational points. Its Frobenius eigenvalues must be equal to $-2$, and
we have
\[
  \sharp Y(\F_{16})=16+1-(-2)^2-(-2)^2=9=\sharp Y(\F_4).
\]
But $Y$
contains $p_1$ and $p_2$ which are defined over $\F_{16}$ but not over
$\F_4$, a contradiction. We deduce that the anticanonical linear
system does not carry any maximal curve over $\F_4$, and the minimum
distance of the anticanonical code is at least $5$. Consequently,
thanks to the Griesmer bound, we get the following result.

\begin{proposition}
\label{max6}
The code $\C(X,-K_X)$ over $\F_4$ has parameters $[13,7,5]$.
\end{proposition}

\section{Anticanonical codes on some degree five del
  Pezzo surfaces}
\label{delpezzo5}

In this section, we construct some degree five del Pezzo surfaces with
Picard rank one over any finite field, then we determine the
parameters of the anticanonical codes on these surfaces. Many
arguments are similar to those of the preceding section, for this
reason we shall skip some proofs.

A new feature here is that we try to be as constructive as possible:
we describe explicit constructions of the codes and their automorphisms.

\subsection{Construction of the surface}

We denote by $\P^2$ the projective plane defined over $\F_q$, and we
choose
\[
  p_1,p_2=p_1^\sigma,p_3=p_1^{\sigma^2},p_4=p_1^{\sigma^3},p_5=p_1^{\sigma^4}
\]
five conjugate points in $\P^2(\F_{q^5})\backslash\P^2(\F_{q})$. We
assume that they are in general position, i.e. no three are
collinear. This is possible for any $q$ since a smooth conic in $\P^2$
has $q^5-q$ points defined over $\F_{q^5}$ but not over $\F_q$.

Let $\widetilde{X}$ be the surface obtained by blowing up the plane
$\P^2$ at the $p_i$'s. Once again we get a degree $4$ del Pezzo surface
and denote by $E_0$ the pullback of the class of a line in $\P^2$ and
by $E_1, \ldots, E_5$ the exceptional divisors. The descriptions of
the geometric Picard lattice of $\widetilde X$, and its intersection
pairing are the same as in Section~\ref{delpezzo6}, as for its
canonical divisor. The difference here is that the map $\sigma^\ast$
acts on the $E_i$'s as the permutation $(E_0)(E_1E_2E_3E_4E_5)$; as a
consequence, the Picard lattice $\Pic(\widetilde{X})$ has rank $2$,
and is generated by $E_0$ and $\sum_{i=1}^5 E_i$.

Let $C$ denote the unique conic passing through $p_1,\ldots,p_5$; it
is defined over $\F_q$. Its strict transform $\widetilde{C}$ in
$\widetilde{X}$ is an irreducible curve, whose class satisfies
\[
  \widetilde{C}=2E_0-\sum_{i=1}^5 E_i \in \Pic(\widetilde{X}).
\]
Once again this curve has self-intersection $-1$ and arithmetic genus zero.

Applying Castelnuovo's contractibility criterion, we obtain a smooth
surface $X$ by contracting the curve $\widetilde{C}$ in
$\widetilde{X}$, and a birational morphism
$\psi:\widetilde{X}\rightarrow X$. If we set $c=\psi(\widetilde{C})$,
then $\psi$ is the blowup of $X$ at $c$, with exceptional divisor
$\widetilde{C}$.

The geometric Picard lattice $\Pic(X\otimes\overline{\F}_q)$ can be
identified to the orthogonal in
$\Pic(\widetilde{X}\otimes\overline{\F}_q)$ of the class of
$\widetilde{C}$ (see Lemma \ref{Picgeom}). After similar calculations,
we get the following ``orthonormal'' basis for
$\Pic(X\otimes\overline{\F}_q)$:
\begin{align*}
  F_0&=3E_0-2E_1-\sum_{i=2}^5E_i;\\
  F_i&=E_0-E_1-E_{i+1},\quad {\rm for} \quad 1\leq i\leq 4  
\end{align*}
These classes satisfy $F_0^{\cdot2}=1$, $F_i^{\cdot2}=-1$ for any
$1\leq i \leq 4$ and $F_i \cdot F_j =0$ for any $i\neq j$.
\begin{remark}\label{rem:contract_Fi}
Note that the classes $F_1, \ldots, F_4$ contain respectively the
strict transforms of the lines $(p_1p_2), \ldots, (p_1p_5)$ and the
corresponding curves satisfy Castel\-nuovo's contractibility criterion.
\end{remark}

The canonical divisor of $X$ satisfies
$\psi^\ast K_X=-5E_0+2\sum_{i=1}^5 E_i$, and we get
$K_X=-3F_0+\sum_{i=1}^4 F_i$ via the above identification. The matrix
of the image $\sigma^\ast$ of Frobenius acting on
$\Pic(X\otimes\overline{\F}_q)$ with respect to the basis
$\{F_0,F_1,F_2,F_3,F_4\}$ is
$$\left(\begin{array}{ccccc}
2 & 1 & 1 & 1 & 0 \\
0 & 0 & 0 & 0 & 1 \\
-1 & 0 & -1 & -1 & 0 \\
- 1 & -1 & 0 & -1 & 0 \\
-1 & -1 & -1 & 0 & 0 \\
\end{array}\right)$$
and has trace zero.

We conclude that $X$ is a degree $5$ del Pezzo surface defined over $\F_q$, with Picard lattice $\Pic(X)$ having rank $1$ and generated by $K_X$. It has $q^2+1$ rational points.

\subsection{Anticanonical codes}
We consider the evaluation code $\C(X(\F_q),-K_X)$.  Its
length is $\sharp X(\F_q)=q^2+1$, and its dimension is at most
$h^0(X,-K_X)=6$.

The anticanonical class $-K_X$ generates $\Pic(X)$. As a consequence,
the sections of its linear system that are defined over $\F_q$ are
irreducible over $\F_q$. Let $D$ denote such a section.
\begin{itemize}
\item If $D$ is absolutely irreducible, then it has arithmetic genus $1$
from the adjunction formula, and has at most
$q+1+\lfloor 2\sqrt{q}\rfloor$ rational points from Hasse-Weil-Serre
bound.
\item If it is not absolutely irreducible, we apply Lemma \ref{fpc}.
\end{itemize}
Once again, the sections having maximal number of zeroes are the
absolutely irreducible ones. Moreover the evaluation map is injective
for any $q\geq 3$, and we deduce the following.

\begin{proposition}

  Assume $q\geq 3$. The code $\C(X(\F_q),-K_X)$ has parameters
  \[[q^2+1,6,\geq q^2-q-\lfloor 2\sqrt{q}\rfloor].\]
\end{proposition}

Parameters of the code for the small values of $q$ are summarised in
Table~\ref{tab:deg5}.

\begin{table}[!h]
  \centering
  \begin{tabular}{|c|c|c|c|c|c|c|}
\hline
$q$ & 3 & 4 & 5 & 7 & 8 & 9 \\
\hline
    $[n,k,d]$ & $[10,6, 3]$ & $[17,6, 8]$ &
              $[26,6, 16]$ & $[50,6, 37]$ &
              $\mathbf{[65,6, 51]}$ & $\mathbf{[82,6, 66]}$\\
\hline
  \end{tabular}
  \bigskip
  \label{tab:deg5}
  \caption{Parameters of the codes from our degree 5 del Pezzo surface}
\end{table}

For $q\in\{4,5,7\}$ the parameters of the best known codes are reached.
For $q\in\{8,9\}$ our codes beat the best known codes listed
in \cite{code}; the best known parameters up to now were respectively
$[65,6,50]_8$ and $[82,6,65]_9$.

\subsection{Geometric constructions of the code}
\label{geo6}

In this section we give two geometric descriptions of the code
$\C(X(\F_q),-K_X)$.

We first use
the anticanonical embedding of the surface $X$. We get a surface in
$\P^5$, and the code can be seen as the code of hyperplane sections on
this surface. Then we construct the code as the puncturing of an
evaluation code at the rational points of the projective space.

\subsubsection{First description}

The class of the anticanonical divisor $-K_X$ in $\Pic(\widetilde{X})$
is $5E_0-2\sum_{i=1}^5 E_i$. Its sections come from quintics in $\P^2$
having double points at the $p_i$'s, $1\leq i\leq 5$. These quintics
form a linear system of (projective) dimension $5$.

We use \cite[Theorem 5]{ghps}, and its proof. We know that this linear
system defines a rational map
\begin{equation*}\label{eq:map_S}
  S:\P^2\dashrightarrow\P^5.
\end{equation*}
Its image is a degree $5$ del Pezzo surface whose ten lines are the
strict transforms of the lines $(p_ip_j)$, $1\leq i<j\leq 5$. The map
$S$ has no unassigned base point, thus its only base points are the
$p_i$, $1\leq i\leq 5$. In order to resolve them, we consider the
surface obtained by their blowup. This is the surface
$\widetilde{X}$. Now the image of $S$ is obtained by contracting the
strict transform of the conic $(p_1\ldots p_5)$, and this is the
anticanonical model of $X$.

As a consequence, the anticanonical code $\C(X(\F_q),-K_X)$ is the
code obtained by evaluating the global sections of the sheaf $\O(1)$
on $\P^5$ (these are linear forms in six variables) at the rational
points of the image of $S$.

\subsubsection{Second description}
We give another description of this code, as the puncturing of an
evaluation code at the rational points of the projective plane.

Consider the six dimensional $\F_q$-linear space $E$ of quintics in
$\P^2$ having double points at the $p_i$'s, $1\leq i\leq 5$. Let $h_C$
denote an equation of the conic $(p_1\ldots p_5)$ in $\P^2$. Then $E$
contains any product $fh_C$, where $f$ is the equation of a cubic
through the $p_i$'s. The space of these cubics has dimension $5$, of
which we fix a basis $\{f_1, \dots, f_5\}$. Finally, we choose a
quintic $Q$ in $E$ not containing $C$. Denote by $h_Q$ an equation of
$Q$. In this way, we get a basis $\B=\{f_1h_C, \ldots, f_5h_C, h_Q\}$
of $E$ over $\F_q$.  Evaluating the elements of $E$ at the rational
points $p\in\P^2(\F_q)$, we get a linear code $\C'$ of length
$q^2+q+1$. The conic $C$ is smooth, and it contains exactly $q+1$
rational points $r_1,\ldots,r_{q+1}$. For any $1\leq i \leq q+1$, the
column corresponding to $r_i$ of the generating matrix of $\C'$
associated to $\B$ has five zeroes at the first positions, and a non
zero coefficient at the last one (note that $Q$ intersects $C$ only at
the $p_i$'s from Bezout's theorem).

From now on, we consider the rational map
\[
  S : \ratmap{\P^2}{\P^5}{P}{(f_1(P)h_C(P):\ldots : f_5(P)h_C(P):h_Q(P))}
\]
associated to the basis $\B$. It has image $X$ from above, and it induces a surjective map on rational
points $\P^2(\F_q)\rightarrow X(\F_q)$. From the description of $S$
above, the restriction of this map to $\P^2(\F_q)\setminus C(\F_q)$ is
injective, and all rational points $r_1,\ldots,r_{q+1}$ of $C$ are
sent to the point $(0:\cdots:0:1)$.

Thus the columns of the generating matrix of the code
$\C(X(\F_q),-K_X)$ are the
$^t(f_1(P)h_C(P)~\ldots~ f_5(P)h_C(P)~ h_Q(P))$, where $P$ describes
$\P^2(\F_q)\setminus C(\F_q)$, and $^t(0\ldots0~1)$ if $P\in
C(\F_q)$. We conclude that the code $\C(X(\F_q),-K_X)$ is obtained
from the code $\C'$ by puncturing the positions corresponding to the
points $r_2,\ldots,r_{q+1}$.

\subsection{Automorphisms}
Our aim here is
to show that the surface $X$ admits an order five automorphism defined
over $\F_q$, and to describe it explicitely.  The proof of the
existence is very close to the proof of Proposition \ref{auto6}, and
we expose it briefly (see also \cite[Theorem 3.1.3]{skoro}).

First denote by $X_0$ the split degree $5$ del Pezzo surface over
$\F_q$ defined by the blowup of the projective plane at the points:
\[
  q_1 := (1:0:0)\quad q_2 := (0:1:0)\quad q_3 := (0:0:1)
  \quad q_4 := (1:1:1).
\]
From \cite[Theorem 8.5.6]{dolga}, the geometric automorphism group of
$X_0$, $\Aut(X_0)$, is isomorphic to the group $W(\E_4)$, i.e. to the
symmetric group on five letters $\mathfrak{S}_5$ with trivial Galois
action (the surface is split).

As a consequence, the set $H^1(G,\Aut(X_0))$ is the set of conjugacy
classes of elements of $\mathfrak{S}_5$. The cocycle associated to the
surface $X$ constructed above sends $\sigma$ to $\gamma$, any of the
(conjugated) order $5$ elements in $\mathfrak{S}_5$. As a consequence,
the geometric automorphisms of $X$ coming from an automorphism defined
over $\F_q$ are exactly those corresponding to the elements in the
centralizer of $\gamma$ in $\mathfrak{S}_5$. But this is the subgroup
generated by $\gamma$.

We construct explicitely an order $5$ automorphism of $X$ defined over $\F_q$; actually we
will show that it is induced by a linear automorphism of $\P^5$. We use 
the preceding discussion. We first determine an isomorphism
$\varphi:X_0\rightarrow X$. Now the automorphisms of the surface $X_0$ are
described in \cite[Section~8.5.4]{dolga} and the automorphism we are looking for is 
conjugated under $\varphi$ to the automorphism of $X_0$ which is the image
of the Frobenius $\sigma$ under the cocycle $c_\varphi$.

Recall that the surface
$\widetilde{X}$ is a degree $4$ del Pezzo surface. We consider two
geometric markings, which lead to the following diagram

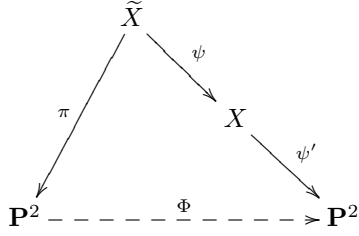
\begin{equation*}
\xymatrix{
& \widetilde{X} \ar@{->}_\pi[ddl] \ar@{->}^\psi[dr]& & \\
& & X \ar@{->}^{\psi'}[dr] \\
\P^2 \ar@{-->}^\Phi[rrr]&  & & \P^2\\
}
\end{equation*}

We get two sequences of contractions from $\widetilde{X}$ to $\P^2$
\begin{itemize}
\item $\pi$ contracts the exceptional divisors $E_1,\dots,E_5$ (the
  first geometric marking),
	\item $\psi$ contracts the exceptional divisor $\widetilde{C}$, and $\psi'$ contracts
      $F_1,\ldots,F_4$, sending them
      respectively to the points $q_3,q_1,q_2,q_4$.
\end{itemize}

Note that the group $\PGL_3$ acts transitively on the $4$-tuples of
points in $\P^2$, which allows the last hypothesis.
 
We get a birational transform $\Phi$ of $\P^2$. Note that $\Phi$ is
not defined over $\F_q$, since $\psi'$ contracts the family of
exceptional divisors $\{F_1,\ldots,F_4\}$ which is not stable under
the Galois action.
The base change matrix (in $\Pic(\widetilde{X}\otimes\overline{\F}_q)$) from  $\{E_0,E_1,\dots,E_5\}$ to $\{F_0,F_1,\ldots,F_4,\widetilde{C}\}$ is the characteristic matrix of $\Phi$ \cite[Example 8.2.38]{dolga}:
$$\left(\begin{array}{cccccc}
3 & 1 & 1 & 1 & 1 & 2 \\
-2 & -1 & -1 & -1 & -1 & -1\\
-1 & -1 & 0 & 0 & 0 & -1 \\
-1 & 0 & -1 & 0 & 0 & -1 \\
-1 & 0 & 0 & -1 & 0 & -1 \\
-1 & 0 & 0 & 0 & -1 & -1 \\
\end{array}\right)$$
From the first column (this is the characteristic of $\Phi$
\cite[7.5.1]{dolga}), we deduce that $\Phi$ has the form
$\Phi(x:y:z)=(h_0:h_1:h_2)$, where $\{h_0,h_1,h_2\}$ is a basis of
the space of cubics having multiplicity at least two at $p_1$ and
passing through $p_2,p_3,p_4,p_5$.

One checks that $\Phi$ contracts the lines $(p_1p_i)$,
$2\leq i\leq 5$, and the conic $C$ (actually, for each
of these curves, the linear system of the above cubics containing it is
one-dimensional). From our assumption on the images of the $F_i$'s,
which are the strict transforms of the lines $(p_1p_{i+1})$,
$1\leq i\leq 4$, we get the following explicit description of $\Phi$,
where $\ell_{ij}$ is the equation of the line $(p_ip_j)$
$$\Phi :(x:y:z)\mapsto (u\ell_{12}\ell_{14}\ell_{35}:
v\ell_{12}\ell_{13}\ell_{45}: w\ell_{13}\ell_{14}\ell_{25})$$ and the
coefficients $u,v,w$ are determined by the condition
$\Phi((p_1p_5))=q_4$. Let $p:=\lambda p_1+\mu p_5$,
$(\lambda:\mu)\in\P^1$, denote any point of $(p_1p_5)$; the point
$\Phi(p)$ has homogeneous coordinates
$$(\lambda\mu^2 u\ell_{12}(p_5)\ell_{14}(p_5)\ell_{35}(p_1):\lambda\mu^2 v\ell_{12}(p_5)\ell_{13}(p_5)\ell_{45}(p_1):\lambda\mu^2 w\ell_{13}(p_5)\ell_{14}(p_5)\ell_{25}(p_1))$$
and we get the following equations defining $u,v,w$ 
$$u\ell_{12}(p_5)\ell_{14}(p_5)\ell_{35}(p_1)=v\ell_{12}(p_5)\ell_{13}(p_5)\ell_{45}(p_1)=w\ell_{13}(p_5)\ell_{14}(p_5)\ell_{25}(p_1).$$

Consider the linear system of cubics passing through the points
$q_1,q_2,q_3,q_4$; we get a map
\[
  S_0 : \ratmap{\P^2}{\P^5}{(x:y:z)}{(y_0:y_1:y_2:y_3:y_4:y_5)}
\]
where
\begin{align*}
  y_0 :=xy(x-z) \qquad
  y_1 &:=xz(x-y)\qquad
  y_2 :=xy(y-z)\\
  y_3 :=yz(y-x)\qquad
  y_4 &:=xz(z-y)\qquad
  y_5 :=yz(z-x)
\end{align*}
and the closure of its image is the anticanonical model of the surface
$X_0$.

Let us make explicit the composition $S_0\circ\Phi$; write it as the
map $(x:y:z)\mapsto(g_0,\ldots,g_5)$. When we compute the first
coordinate, we get
$$ u\ell_{12}\ell_{14}\ell_{35}v\ell_{12}\ell_{13}\ell_{45}(\ell_{14}(u\ell_{12}\ell_{35}-w\ell_{13}\ell_{25})).$$
Now if $p:=\lambda p_1+\mu p_5$, $(\lambda:\mu)\in\P^1$, is any point
of $(p_1p_5)$, we have
\[
  u\ell_{12}(p)\ell_{35}(p)-w\ell_{13}(p)\ell_{25}(p)=\lambda\mu(u\ell_{12}(p_5)\ell_{35}(p_1)-w\ell_{13}(p_5)\ell_{25}(p_1))=0
  \]
from the equations defining $u,v,w$, and $\ell_{15}$ divides
$u\ell_{12}\ell_{35}-w\ell_{13}\ell_{25}$. As this last polynomial
defines a conic passing through $p_2$ and $p_3$, we must have
$u\ell_{12}\ell_{35}-w\ell_{13}\ell_{25}=\alpha_0\ell_{15}\ell_{23}$
for some $\alpha_0\in \overline{\F}_q$.

When we apply the same calculations to the $g_i$'s, $1\leq i\leq 5$,
and simplify by $\ell_{12}\ell_{13}\ell_{14}\ell_{15}$, we get
\begin{align*}
  g_0&=\alpha_0 u v \ell_{12}\ell_{14}\ell_{23}\ell_{35}\ell_{45},\qquad
g_1=\alpha_1 u w \ell_{12}\ell_{14}\ell_{25}\ell_{35}\ell_{34},\\
g_2&=\alpha_2 u v \ell_{12}\ell_{13}\ell_{24}\ell_{35}\ell_{45},\qquad
g_3=-\alpha_1 v w \ell_{12}\ell_{13}\ell_{34}\ell_{25}\ell_{45},\\
g_4&=-\alpha_2 u w \ell_{13}\ell_{14}\ell_{24}\ell_{25}\ell_{35},\quad           
g_5=-\alpha_0 v w \ell_{13}\ell_{14}\ell_{23}\ell_{25}\ell_{45}
\end{align*}
where $\alpha_1,\alpha_2$ are defined by
$$\alpha_1\ell_{15}\ell_{34}=u\ell_{14}\ell_{35}-v\ell_{13}\ell_{45},~\alpha_2\ell_{15}\ell_{24}=v\ell_{12}\ell_{45}-w\ell_{14}\ell_{25}.$$

Observe that $\{g_0,\dots,g_5\}$ is a new basis for the
$\overline{\F}_q$-vector space of quintics passing through
$p_1,\dots,p_5$ with multiplicity at least two. As a consequence,
there is an element $M$ in $\GL_6(\overline{\F}_q)$ that sends the
basis $\{g_0,\dots,g_5\}$ to $\B\{f_1h_C,\ldots,f_5h_C,h_Q\}$, and it
induces a linear automorphism of $\P^5$ whose restriction to $X_0$ is
an isomorphism $\varphi:X_0\rightarrow X$; we get a diagram

\begin{equation*}
\xymatrix{
 X  & \ar@{->}_{\varphi}[l] X_0 \\
\P^2 \ar@{-->}^{\Phi}[r] \ar@{-->}^{S}[u] &  \P^2 \ar@{-->}_{S_0}[u] \\
}
\end{equation*}

Let us write the incidence graphs of the surfaces $X$ and $X_0$
\cite[Section~8.5.1]{dolga}, \cite[26.9]{manin}. The exceptional lines
on $X$ correspond to the following ten classes
\[
  l_{ij}:=E_0-E_i-E_j, \quad 1\leq i<j\leq 5
\]
in $\Pic(X\otimes\overline{\F}_q)$ (the notation $l_{ij}$ is chosen to
recall that $l_{ij}$ is the image in $X$ of the strict transform in
$\widetilde{X}$ of the line $(p_ip_j)$ in $\P^2$). The exceptional
divisors of $X_0$ are labelled from the points $q_i$ and lines
$L_{ij}:=(q_iq_j)$, $1\leq i<j\leq 4$ of $\P^2$ corresponding to the
exceptional divisors on $X_0$.

Both are Petersen graphs, and the automorphism group
$\Aut_{\overline{\F}_q}(X)$ is isomorphic to the group of symmetries
of these graphs \cite[Lemma 13]{ghps}.
\begin{center}

\begin{tikzpicture}
  \graph[clockwise, radius=2cm] {subgraph C_n
    [V={$l_{25}$,$l_{14}$,$l_{35}$,$l_{24}$,$l_{13}$}, name=A]};
  \graph[clockwise, radius=1cm] {subgraph I_n
    [V={$l_{34}$,$l_{23}$,$l_{12}$,$l_{15}$,$l_{45}$}, name=B]};

  \draw (B $l_{34}$) -- (B $l_{12}$);
	\draw (B $l_{23}$) -- (B $l_{15}$);
	\draw (B $l_{12}$) -- (B $l_{45}$);
	\draw (B $l_{15}$) -- (B $l_{34}$);
	\draw (B $l_{45}$) -- (B $l_{23}$);
	\draw (A $l_{25}$) -- (B $l_{34}$);
	\draw (A $l_{14}$) -- (B $l_{23}$);	
	\draw (A $l_{35}$) -- (B $l_{12}$);
	\draw (A $l_{24}$) -- (B $l_{15}$);
	\draw (A $l_{13}$) -- (B $l_{45}$);
\end{tikzpicture}	
\hspace{1cm}
\begin{tikzpicture}
  \graph[clockwise, radius=2cm] {subgraph C_n
    [V={$L_{12}$,$q_2$,$L_{23}$,$L_{14}$,$q_1$}, name=A]};
  \graph[clockwise, radius=1cm] {subgraph I_n
    [V={$L_{34}$,$L_{24}$,$q_3$,$q_4$,$L_{13}$}, name=B]};

  \draw (B $L_{34}$) -- (B $q_3$);
	\draw (B $L_{24}$) -- (B $q_4$);
	\draw (B $q_3$) -- (B $L_{13}$);
	\draw (B $q_4$) -- (B $L_{34}$);
	\draw (B $L_{13}$) -- (B $L_{24}$);
	\draw (A $L_{12}$) -- (B $L_{34}$);
	\draw (A $q_2$) -- (B $L_{24}$);	
	\draw (A $L_{23}$) -- (B $q_3$);
	\draw (A $L_{14}$) -- (B $q_4$);
	\draw (A $q_1$) -- (B $L_{13}$);
\end{tikzpicture}	

\end{center}

Now $\varphi$ sends each vertex of the right-hand graph to the vertex
at the same place of the left-hand one. As a consequence, since the
Frobenius automorphism acts as the rotation with angle
$\frac{2\pi}{5}$ on the left-hand graph, the cocycle associated to
$\varphi$ sends it to the automorphism acting as the rotation with
angle $\frac{2\pi}{5}$ on the right-hand one.

The automorphism of $X_0$ acting as above on the graph is given in
\cite[Section~8.5.4]{dolga}; it comes from the birational map of
$\P^2$ (or Cremona transformation)
\[
  \delta:(x:y:z)\mapsto(x(z-y):z(x-y):xz),
\]
whose action induces an
automorphism of the anticanonical model of $X_0$ that comes from the
following linear action
$$D:(y_0,y_1,y_2,y_3,y_4,y_5)\mapsto(y_0-y_1-y_2,y_0-y_1+y_4,-y_3-y_4+y_5,y_1-y_4+y_5,y_4,y_1).$$

From our calculations in Galois cohomology at the beginning of the
section, we see that the automorphism of $X$ conjugated to $D$ under
the action of $\varphi$ is defined over $\F_q$. Summarizing, we get
the following diagram

\begin{equation*}
\xymatrix{
\P^2 \ar@{-->}_S[d]  \ar@{-->}^\Phi[r] &   \P^2  \ar@{-->}^{S_0}[d] \ar@{-->}^\delta[r] &  \P^2 \ar@{-->}_{S_0}[d]  & \ar@{-->}_\Phi[l] \P^2 \ar@{-->}^S[d]\\
\P^5  \ar@/_/[rrr]_A & \ar@{->}_{M}[l] \P^5 \ar@{->}^{D}[r] & \P^5 \ar@{->}^{M}[r] & \P^5\\
}
\end{equation*}
and an order five automorphism of $X$ defined over $\F_q$ acts on its
anticanonical model by the restriction of the linear action
$A=M^{-1}\circ D \circ M$ on $\P^5$.

\section{Anticanonical codes on some degree four del Pezzo surfaces}
\label{degre4}

In this section, we focus on degree four del Pezzo surfaces and
especially those with Picard rank equal to one.

\subsection{Construction of the surface}

Over~$\overline{\F}_q$, del Pezzo surfaces of degree~$4$ are all the
blowing up of~$\P^2$ in five points in general position. As in
degrees~$5$ or~$6$, this model may not be defined over~$\F_q$. Instead
of computing a birational and $\F_q$-rational morphism from~$\P^2$
to the considered degree~$5$ or~$6$ del Pezzo surface, we adopt a
different strategy in degree~$4$. In fact, it turns out that,
following Flynn \cite{fly}, one can directly compute the
anti-canonical model of a degree~$4$ del Pezzo surface from the
Frobenius action on the geometric Picard group, at least when the
characteristic is odd.  This model, which is defined over the base
field, is embedded in~$\P^4$ as the intersection of two quadrics.
% Nevertheless, only invariant of the type, such as Picard rank and
% the number of points, are important for code applications.
The starting point is still a type of Frobenius action but we have to
observe this action from a different point of view to understand
Flynn's construction.

So, let~$X$ be a degree~$4$ del Pezzo surface. Geometrically, it is
the blowup of~$\P^2$ at five points in general position and as before,
we denote by~$E_0$ the pullback of the class of a line in~$\P^2$ and
by~$E_1,\ldots,E_5$ the five exceptional divisors. One can prove
that~$X$ contains exactly ten families of conics whose classes
are
\[
  C_i = E_0-E_i,
  \quad {\rm and}\quad
  C'_i = -K_X - C_i = 2E_0 -\sum_{j\not= i} E_j,\quad
  {\rm for} \quad1\leq i\leq 5
\]
where~$K_X = -3E_0 + \sum_{i=1}^5 E_i$ is the
canonical class \cite[\S2,Th~2]{BBFL}. These classes are the only
ones satisfying the intersection constraints
\begin{align*}
  &C\cdot(-K_X) = 2
  &
  &\text{and}
  &
  &C^2 = 0.
\end{align*}
So the Frobenius~$\sigma^*$ acts on these classes and one can recover
the Frobenius action on the whole
space~$\Pic(X\otimes \overline{\F}_q)$ from this action since one
easily checks that the family composed by the
classes~$\frac{1}{2}\left(-K_X +\sum_{i=1}^5 C_i\right)$ and
the~$C_i$'s for~$1\leq i\leq 5$ is a basis
of~$\Pic(X\otimes \overline{\F}_q)$.

This action has a strong geometric flavour that we now
describe. Let~$Q_1,Q_2$ be two quadrics of~$\P^4$ whose intersection
defines the anti-canonical embedding of the surface~$X$. Since~$X$ is
smooth, in the pencil of quadrics of~$\P^4$ containing~$X$, i.e.
the~$\lambda_1 Q_1 - \lambda_2 Q_2$
for~$(\lambda_1:\lambda_2)\in\P^1(\overline{\F}_q)$, there are exactly
five singular quadrics \cite[\S3.3]{Wittenberg}. The
points~$(\lambda_1:\lambda_2)\in\P^1(\overline{\F}_q)$ corresponding
to these singular quadrics are conjugate under~$G$ since they are the
roots of the determinant~$\det(\lambda_1 S_1 - \lambda_2 S_2)$,
where~$S_i$ is the $5\times 5$ symmetric matrix associated to the
quadratic form~$Q_i$.  The intersection of each singular quadric with
its tangent space at a smooth point is the union of two planes.  It
can be shown that the classes of the intersections of these two planes
with the surface~$X$ are equal to~$C_i$ and~$C'_i$ for
some~$1\leq i\leq 5$ \cite[\S2,Th~4]{BBFL}. In other terms, to each
singular quadric in the pencil containing~$X$ corresponds a
pair~$\{C_i,C'_i\}$.

From the Galois point of view, we deduce how~$G$
acts on the ten conics~$C_i,C'_i$, $1\leq i\leq 5$
\cite[\S2.4]{VarillyViray}. First~$G$ acts on the set of
pairs~$\{C_i,C'_i\}$ as it acts on the five conjugates points
of~$\P^1(\overline{\F}_q)$ associated to the five singular quadrics in
the pencil. More precisely,
if~$(\lambda_1:\lambda_2)\in\P^1(\overline{\F}_q)$ is one of these
points, then the action on the subset of the~$C_i,C'_i$'s
corresponding to~$(\lambda_1:\lambda_2)$ and all its conjugates is
transitive if and only if the two previous planes are not defined over
the field of definition~$\F_q\left((\lambda_1:\lambda_2)\right)$ (but
over the quadratic extension of it).  In brief, the characterisation
of the Frobenius action on the ten classes of conics~$C_i,C'_i$,
$1\leq i\leq 5$ can be reduced to a
sequence~$d_1[\epsilon_1]\cdots d_r[\epsilon_r]$ with~$d_i$ positive
integers satisfying~$\sum_i d_i = 5$ and
with~$\epsilon_i\in\{\pm 1\}$.  Let us call this data the type of the
action. There are three types that lead to a surface with Picard rank
equal to one which are listed in Table~\ref{tab:degree4}
(see ~\cite[Table~3]{DolgachevDuncan}):

\begin{table}[!h]
  \centering
  \begin{tabular}{|c|c|c|c|c|}
\hline
Type & Frobenius action & Eigenvalues of $\sigma^\ast$ & $\Tr (\sigma^*)$ & Picard rank \\
\hline
\hline
$4_1$ & $2[-1]1[-1]1[-1]\,1[-1]$ & $1, -1, -1, -1, i, -i$ & -2 & 1 \\
\hline
    $4_2$ & $4[-1]1[-1]$ & $1, -1, \zeta_8=e^{i\pi/ 8}, \zeta_8^3, \zeta_8^5,
                           \zeta_8^7$ & 0 & 1\\
\hline
$4_3$ & $3[-1]2[-1]$ & $1, -1, i, -i, -\jmath, -\jmath^2$ & 1 & 1\\
\hline
  \end{tabular}
  \bigskip
  \caption{Types of degree 4 del Pezzo surfaces with Picard rank 1.}
  \label{tab:degree4}
\end{table}

Unlike the degrees~$5$ or~$6$ cases, there are several isomorphism
classes in every type of Frobenius action.  Flynn develops a very
powerful method which, given a type of Frobenius action in the
previous sense, computes the anti-canonical model of a del Pezzo
surface having this type of action \cite{fly,Skoro_Del4}. It works as
follows, starting from a type~$d_1[\epsilon_1]\cdots
d_r[\epsilon_r]$. One has to choose $r$ distinct irreducible
polynomials~$f_1, \ldots, f_r\in\F_q[x]$ of respective
degrees~$d_1,\ldots,d_r$, and for each~$1\leq i\leq r$ a non zero
element~$\delta_i\in\F_q[x]/\langle f_i\rangle$ which is a square or
not depending on whether~$\epsilon_i = 1$ or~$-1$. Then, we
put~$f=f_1\cdots f_r$ and we
define~$\delta\in\F_q[x]/\langle f\rangle$ in such a way that~$\delta$
is sent to~$(\delta_1,\ldots,\delta_r)$ by the Chinese Remainder
isomorphism between~$\F_q[x]/\langle f\rangle$ and the
product~$\prod_i\F_q[x]/\langle f_i\rangle$. Formally,
in~$\F_q[x]/\langle f\rangle$, we compute the five
quadrics~$Q_0,\ldots,Q_4$ in~$x_0,\ldots,x_4$ such that~:
$$
\delta \times \left(x_0 + x_1 x + \dots + x_4 x^4\right)^2
=
%Q_{0}(x_0,\ldots,x_4) + Q_{1}(x_0,\ldots,x_4)x + Q_{2}(x_0,\ldots,x_4)x^2 + Q_{3}(x_0,\ldots,x_4)x^3 + Q_{4}(x_0,\ldots,x_4)x^4
Q_{0}(x_0,\ldots,x_4) + \dots + Q_{4}(x_0,\ldots,x_4)x^4
$$
Then the surface~$X$ in~$\P^4$, defined by~$Q_{3}(x_0,\ldots,x_4) = Q_{4}(x_0,\ldots,x_4) = 0$ is a del Pezzo surface of given type.

\subsection{Anticanonical codes}

We now determine the parameters of the evaluation
code~$\C(X(\F_q),-K_X)$ when~$X$ is a del Pezzo surface of degree~$4$
of Picard rank one. We denote by~$\Tr(\sigma^*)$ the trace of the
Frobenius morphism, which, as we have noticed, is an element
of~$\{-2,0,1\}$ in our cases.  The global sections of the
sheaf~$\O(-K_X)$ are nothing else than the five coordinate
functions~$x_0,\ldots,x_4$ and the dimension of the code is thus
five. As for the length, it is the number of rational points of~$X$,
which is given by~$q^2+1 + q\Tr(\sigma^*)$.  The construction of a
generator matrix is obvious and only consists in the vertical join of
the coordinates vector of the rational points of the surface~$X$:
\begin{align*}
&\begin{pmatrix}
  x_0(p_1) & \cdots & x_0(p_N)\\
  \vdots   &        & \vdots\\
  x_4(p_1) & \cdots & x_4(p_N)
\end{pmatrix},
&
&X(\F_q) = \left\{p_1,\ldots,p_N\right\}.
\end{align*}
Finally, let us compute the minimal distance.  Since~$-K_X$ is a
generator of~$\Pic(X)$, any effective divisor~$D \in |-K_X|$ must be
irreducible over~$\F_q$. If~$D$ is absolutely irreducible then the
(arithmetic) genus~$\pi(D)$ of~$D$ is~$1$ since
\[
  2\pi(D)-2 = K_X\cdot(D+K_X) = K_X\cdot (-K_X+K_X) = 0.
\]
In that case,
one has
\[
  \sharp D(\F_q) \leq q+1 + \lfloor 2\sqrt{q}\rfloor.
\]
If~$D$ is not absolutely irreducible, then from Lemma~\ref{fpc}, we
have~$\sharp D(\F_q) \leq 2$.  In any case, the number of
rational points of~$D$ is bounded above
by~$q+1 + \lfloor 2\sqrt{q}\rfloor$. We deduce that:

\begin{proposition}\label{prop_del4}
Assume~$q$ odd. Then, the code~$\C(X(\F_q),-K_X)$ has parameters
$$
[n,k,d]
=
\begin{cases}
  [q^2-2q+1, 5, q^2 - 3q - \lfloor 2\sqrt{q}\rfloor] & \text{for the type~$4_1$;}\\
  [q^2+1, 5, q^2 - q - \lfloor 2\sqrt{q}\rfloor] &\text{for the type~$4_2$;}\\
  [q^2+q+1, 5, q^2 - \lfloor 2\sqrt{q}\rfloor] & \text{for the
    type~$4_3$.}
\end{cases}
$$
\end{proposition}

Parameters of such codes for small values of $q$ are listed in
Table~\ref{tab:param_degree4}.
\begin{table}[!h]
  \centering
  \begin{tabular}
    {|c|c|c|c|c|c|}
\hline
    &$q$&3&5&7&9\\
\hline
    Type~$4_1$ & 
    $\left[n,k,d\right]$&      & $[16,5,6]$ & $[36,5,23]$ & $[64,5,48]$ \\
\hline
    Type~$4_2$ & 
    $\left[n,k,d\right]$ & $[10,5,3]$ & $[26,5,16]$ & $[50,5,37]$ & $[82,5,66]$ \\
\hline
   Type~$4_3$ & 
   $\left[n,k,d\right]$& $[13,5,6]$& $[31,5,21]$ & $[57,5,44]$ & $[91,5,75]$\\
\hline
  \end{tabular}
  \bigskip
  \caption{Parameters of codes from degree 4 del Pezzo surfaces of rank 1}
  \label{tab:param_degree4}
\end{table}

The codes of the line ``Type~$4_3$'' all attain the parameters of the
best known codes \cite{code}.

\begin{remark}
  For~$q=3$ the ``Type~$4_1$'' does not exist, since there are not
  enough non squares in~$\F_3$; see also \cite[Theorem 1.4 (1)]{trepa}.  
\end{remark}

\begin{remark}	
	In fact, the parameters given in
  proposition~\ref{prop_del4} are also true in even characteristic, at
  least if the type exists.  But, on the contrary the Flynn's
  construction only works in odd characteristic.
  Over~$\F_8 = \F_2(\zeta)$, by a simple random search, we have found
  the two quadrics
\begin{align*}
Q_1=\zeta^2x_0^2 &+ \zeta^5x_0x_1 + \zeta^4x_1^2 + \zeta x_0x_2 + \zeta^6x_1x_2 + \zeta x_2^2 + \zeta^3x_0x_3\\
            &+ \zeta^4x_1x_3 + \zeta^3x_2x_3 + \zeta x_3^2 + x_0x_4 + \zeta^6x_1x_4 + \zeta^6x_2x_4 + x_3x_4
\\
Q_2=\zeta^3x_0^2 &+ \zeta^5x_0x_1 + \zeta^2x_1^2 + \zeta^4x_2^2 + \zeta^5x_0x_3 + \zeta^2x_1x_3 + x_2x_3 + \zeta^5x_3^2\\
            &+ \zeta^4x_0x_4 + \zeta x_1x_4 + \zeta^6x_2x_4 + \zeta x_3x_4 + \zeta^3x_4^2
\end{align*}
that define a del Pezzo surface of degree~$4$ over~$\F_8$, of
``Type~$4_3$'', whose anticanonical code has
parameters~$[n,k,d] = [ 73, 5, 59 ]$.
\end{remark}

Let us end this section by displaying a complete example in the
type~$4_3$. We use {\tt magma} for the computations. We start form two
polynomials~$f_2,f_3\in\F_q[x]$ of degrees~$2$ and~$3$ such
that~$x\bmod{f_2}$ and~$x\bmod{f_3}$ are non squares in~$\F_{q^2}$
and~$\F_{q^3}$. Then the choice~$\delta = x$ is good. For example,
for~$q=5$, the standard polynomials, $f_2(x)=x^2+4x+2$
and~$f_3(x)=x^3+3x+3$ chosen by {\tt magma} to construct~$\F_{5^2}$
and~$\F_{5^3}$, suit. Their product is~$f(x)=x^5+4x^4+3x+1$. Using
{\tt magma}, we easily compute the two associate quadrics
in~$\F_5[x_0,\ldots,x_4]$
\begin{align*}
Q_3 = x_1^2 &+ 2x_0x_2 + 2x_3^2 + 4x_2x_4 + 2x_3x_4 + x_4^2\\
Q_4 = 2x_1x_2 &+ x_2^2 + 2x_0x_3 + 2x_1x_3 \\&+ 2x_2x_3 + x_3^2 + 2x_0x_4 + 2x_1x_4 + 2x_2x_4 + x_3x_4 + 4x_4^2
\end{align*}
The subvariety of~$\P^4$ defined by~$Q_1=Q_2=0$ is a degree~$4$ del
Pezzo surface of type~$4_3$, which contains $31=5^2+5+1$ points.
Joining these points in columns, we obtain the generator matrix of the
code~$\C(X(\F_5),-K_X)$ over~$\F_5$:
$$
G
=
\begin{pmatrix}
1\,3\,0\,4\,3\,3\,1\,3\,1\,2\,0\,4\,2\,1\,2\,3\,3\,0\,3\,2\,3\,0\,1\,0\,2\,1\,1\,4\,2\,3\,1\\
0\,0\,0\,0\,0\,0\,0\,0\,0\,1\,1\,1\,1\,2\,2\,2\,2\,3\,3\,3\,3\,4\,4\,4\,4\,4\,1\,2\,3\,2\,0\\
0\,0\,1\,2\,2\,2\,4\,4\,4\,1\,1\,2\,3\,0\,1\,2\,3\,0\,1\,2\,2\,1\,1\,1\,3\,4\,1\,3\,1\,1\,0\\
1\,3\,4\,0\,1\,3\,0\,1\,4\,0\,3\,3\,1\,0\,2\,0\,4\,0\,0\,1\,4\,1\,2\,3\,1\,1\,1\,1\,1\,0\,0\\
1\,1\,1\,1\,1\,1\,1\,1\,1\,1\,1\,1\,1\,1\,1\,1\,1\,1\,1\,1\,1\,1\,1\,1\,1\,1\,0\,0\,0\,0\,0
%1&3&0&4&3&3&1&3&1&2&0&4&2&1&2&3&3&0&3&2&3&0&1&0&2&1&1&4&2&3&1\\
%0&0&0&0&0&0&0&0&0&1&1&1&1&2&2&2&2&3&3&3&3&4&4&4&4&4&1&2&3&2&0\\
%0&0&1&2&2&2&4&4&4&1&1&2&3&0&1&2&3&0&1&2&2&1&1&1&3&4&1&3&1&1&0\\
%1&3&4&0&1&3&0&1&4&0&3&3&1&0&2&0&4&0&0&1&4&1&2&3&1&1&1&1&1&0&0\\
%1&1&1&1&1&1&1&1&1&1&1&1&1&1&1&1&1&1&1&1&1&1&1&1&1&1&0&0&0&0&0
\end{pmatrix}
$$
This code has parameters~$[31,5,21]$, as confirmed by {\tt magma}.

\bibliographystyle{amsalpha}
\bibliography{AnticanonicalCodes}

\end{document}